\documentclass[11pt,reqno]{amsart}
\usepackage{fullpage}
\usepackage{amsfonts}
\usepackage{amssymb}
\usepackage{times}
\usepackage{graphicx}
\usepackage{amsmath,bm}
\usepackage{tikz-cd}
\usepackage{tikz}
\usetikzlibrary{arrows,decorations.pathreplacing}

\usepackage{appendix}
\usepackage{color}
\usepackage{mathrsfs}
\usepackage{geometry}
\newtheorem{thm}{Theorem}[section]
\newtheorem{cor}[thm]{Corollary}
\newtheorem{lem}[thm]{Lemma}
\newtheorem{prop}[thm]{Proposition}

\theoremstyle{definition}
\newtheorem{ex}[thm]{Example}
\newtheorem{nota}[thm]{Notation}
\newtheorem{defn}[thm]{Definition}

\theoremstyle{remark}
\newtheorem{rem}[thm]{Remark}

\newcommand{\G}{\Gamma}
\newcommand{\e}{{\epsilon}} 
\newcommand{\mcP}{\mathcal P}
\newcommand{\mcD}{\mathcal D}
\newcommand{\mcU}{\mathcal U}
\newcommand{\wmcU}{\widetilde {\mathcal U} }

\newcommand{\mB}{{\mathbb B}}

\newcommand{\mG}{{\mathbb G}}

\newcommand{\mR}{{\mathbb R}}

\newcommand{\mT}{{\mathbb T}}
\newcommand{\mZ}{{\mathbb Z}}

\newcommand{\man}{{\mathrm{an}}}
\newcommand{\mtrop}{{\mathrm{Trop}}}
\newcommand{\mpic}{{\mathrm{Pic}}}

\begin{document}

\title[Short version of title]{Lifting Divisors with Imposed Ramifications\\ on a Generic Chain of Loops}
\author{Xiang He}
\maketitle
\begin{abstract}
Let $C$ be a curve over an algebraically closed non-archimedean field with non-trivial valuation. Suppose $C$ has totally split reduction and the skeleton $\G$ is a chain of loops with generic edge lengths. Let $P$ be the rightmost vertex of $\G$ and $\mcP\in C$ be a point that specializes to $P$. We prove that any divisor class on $\G$ with imposed ramification at $P$ that is rational over the value group of the base field lifts to a divisor class on $C$ that satisfies the same ramification at $\mcP$, which extends the result in \cite{cartwright2014lifting}. 
\end{abstract}

\section{Introduction}
A metric graph which is a generic chain of loops (Definition \ref{generic chain}) plays a crucial role in connecting classic and tropical Brill-Noether theory. Many properties of these graphs, such as Brill-Noether generality, can be transferred to certain curves with minimal skeleton isometric to them. Related approaches can be found in \cite{cools2012tropical,jensen2014tropical,jensen2016tropical}. 

Let $\G$ be a generic chain of loops with or without bridges. Let $K$ be an algebraically closed non-archimedean field with nontrivial value group $G$ and valuation ring $R$. 
Let $C$ be a smooth projective curve of genus $g$ over $K$ which has totally split reduction (by which we mean $C$ admits a split semistable $ R$-model as in \cite[\S 5]{baker2014skeleton} whose special fiber only has rational components) and the skeleton is isometric to $\Gamma$. The tropicalization map from $\mpic^d(C)$ to $\mpic^d(\G)$ given by linear expansion of the retraction from $C^\man$ to $\G$ maps $W^r_d(C)^\man$ to $W^r_d(\G)$ \cite{baker2008specialization}, where $W^r_d(C)$ (resp. $W^r_d(\G)$) parametrizes divisor classes on $C$ (resp. $\G$) with degree $d$ and rank at least $r$. It is then proved in \cite{cartwright2014lifting} that $\mtrop(W^r_d(C))=W^r_d(\G)$ via the classification of divisors in $W^r_d(\G)$, given in \cite{cools2012tropical}. In other words, every $G$-rational divisor class on $\G$ of rank $r$ can be lifted to a divisor class on $C$ of the same rank.

On the other hand, let $\alpha=(\alpha_0,...,\alpha_r)$ be a \textbf{Schubert index} of type $(r,d)$, which is a non-decreasing sequence of non-negative numbers bounded by $d-r$. For an arbitrary chain of loops $\G$, 
Pflueger \cite{pflueger2017special} provides a straightforward expression for the locus of divisors on $\G$ with ramification at least $\alpha$ at the rightmost vertex $P$ of $\G$:
$$W^{r,\alpha}_d(\G,P)=\{D\in\mathrm{Pic}^d(\G): r( D-(\alpha_i+i)P)\geq r-i\mathrm{\ for\ } i=0,1,...,r\},$$ which is locally a union of translates of coordinate planes possibly of different dimensions 
 in $\mR^g$. 
It follows that $W^{r,\alpha}_d(\G,P)$ contains the tropicalization of the corresponding locus on $C$:
$$W^{r,\alpha}_d(C,\mcP)=\{\mcD\in\mathrm{Pic}^d(C): h^0(\mathscr O_C(\mcD-(\alpha_i+i)P))\geq r+1-i\mathrm{\ for\ } i=0,1,...,r\}.$$ 
When $\G$ is a generic chain of loops, \cite{pflueger2017special} shows that  $W^{r,\alpha}_d(\G,P)$ has dimension equal to that of $W^{r,\alpha}_d(C,\mcP)$, and both of pure dimensions $\rho(g,r,d)-\sum_{0\leq i\leq r}\alpha_i$ as expected. We prove an analogue of the lifting result of \cite{cartwright2014lifting}:

\begin{thm}\label{introduction}
Let $\G$ be a generic chain of loops, possibly with bridges, and $C$ a smooth projective curve over $K$ which has totally split reduction and skeleton isometric to $\G$. Let $\mcP$ be a point on $C$ that tropicalize to the rightmost vertex $P$ of $\G$. Let $\alpha$ be a Schubert index of type $(r,d)$. 
Then every $G$-rational divisor class on $\G$ with ramification at least $\alpha$ at $P$ lifts to a divisor class on $C$ with ramification at least $\alpha$ at $\mcP$.
\end{thm}
 

We now explain the general strategy used in the proof of Theorem \ref{introduction}. Let $$\alpha^j=(\alpha_0,...,\alpha_j,\alpha_j,...,\alpha_j)$$be a Schubert index of type $(r,d)$ whose last $r-j+1$ coordinates are all $\alpha_j$s. Denote $W_j(C)=W^{r,\alpha^j}_d(C,\mcP)$ and $W_j(\G)=W^{r,\alpha^j}_d(\G,P)$. Denote also $$X_j(C)=(\alpha_j+j)\mcP+W^{r-j}_{d-\alpha_j-j}(C)\mathrm{\ and\ } Y_j(C)=(\alpha_j+j+1)\mcP+W^{r-j-1}_{d-\alpha_j-j-1}(C)$$
 $$X_j(\G)=(\alpha_j+j)P+W^{r-j}_{d-\alpha_j-j}(\G)\mathrm{\ and\ }Y_j(\G)=(\alpha_j+j+1)P+W^{r-j-1}_{d-\alpha_j-j-1}(\G)$$ 
Note that
$W_j(C)
=W_{j-1}(C)\cap X_j(C)$
and that 
$$
W_{j}(\G)
=W_{j-1}(\G)\cap X_j(\G)
=W_{j-1}(\G)\cap\mtrop (X_j(C)).
$$
We proceed by induction. At each step, assuming  
$\mtrop (W_{j-1}(C))=W_{j-1}(\G)$,
it suffices to show that   
\begin{equation}\mtrop(W_{j-1}(C)\cap X_j(C))=\mtrop(W_{j-1}(C))\cap \mtrop(X_j(C)).\end{equation}
Note that $Y_{j-1}(C)$ is locally isomorphic to a \textbf{polytopal domain}, which is the preimage of an integral $G$-affine polytope in $\mR^n$ of the tropicalization map $(\mG_m^n)^\man\rightarrow \mR^n$ for some $n$, at points in the relative interior of maximal faces of $Y_{j-1}(\G)=\mtrop(Y_{j-1}(C))$.
Since $W_{j-1}(\G)$ and $X_j(\G)=\mtrop(X_j(C))$ \textbf{intersect properly} in $Y_{j-1}(\G)$, namely intersect in expected dimension, the problem boils down to lifting proper tropical intersections within a polytopal domain:

\begin{thm}\label{lifting polyhedral domain}
Let $\mathcal U_\Delta \subset (\mathbb G_m^n)^{\mathrm{an}}$ be the preimage of an integral $G$-affine polytope $\Delta$ in $\mR^n$ of dimension $n$ and $\mathcal X$ and $\mathcal X'$ be two Zariski closed analytic subspaces of $\mathcal U_\Delta$ of pure dimension. Suppose $\mtrop(\mathcal X)$ and $\mtrop(\mathcal X')$ intersect properly in $\Delta$. Then we have $$\mathrm{Trop}(\mathcal X)\cap\mathrm{Trop}(\mathcal X')\cap\Delta^\circ=\mathrm{Trop}(\mathcal X\cap\mathcal X')\cap \Delta^\circ.$$
\end{thm}

This theorem is proved in Section 4. See also \cite{osserman2013lifting} for an algebraic counterpart, where the authors proved a lifting theorem for subschemes of an algebraic torus whose tropicalizations intersect properly.  See Section 3 for the discussion of local property of $Y_{j-1}(C)$, which works for all Brill-Noether loci $W^{r'}_{d'}(C)$ on $C$. The proof of (1) is in Section 5, where we use the notation of ramification imposed by a partition instead of a Schubert index (Section 2). 


\subsection*{Conventions.} Throughout this paper $K$ will be an algebraically closed non-archimedean field $K$ with nontrivial value group $G$. For a lattice $N$ we denote $T_N$ the algebraic torus over $K$ whose lattice of characters, denoted by $M$, is dual to $N$. Denote also $N_\mR=N\otimes \mR$. All polytopes in $N_\mR$ are assumed integral $G$-affine.

\subsection*{Acknowledgements.} I would like to thank Brian Osserman for helpful conversations and Sam Payne for 
suggesting that the analytic continuity theorem of intersection numbers (\cite[Proposition 5.8]{osserman2011lifting}) may lead to a proof of Theorem \ref{lifting polyhedral domain}.

\section{preliminaries}
In this section we recall some notions and techniques which are useful for later arguments.
\subsection{Special divisors on a generic chain of loops}
Let $\Gamma$ be a metric graph that is a chain of $g$ loops with or without bridges, let $\{v_i\}_{1\leq i \leq g}$ and $\{w_i\}_{1\leq i\leq g}$ be vertices of $\Gamma$ as in Figure \ref{fig1} (with the possibility that $w_i=v_{i+1}$). Let $l_i$ (resp. $n_i$) be the length of the top (resp. bottom) segment of the $i$th loop connecting the vertices $v_i$ and $w_i$. 
The divisors on $\G$ with imposed ramification is classified in \cite{pflueger2017special}, we recall some related concepts from loc.cit..


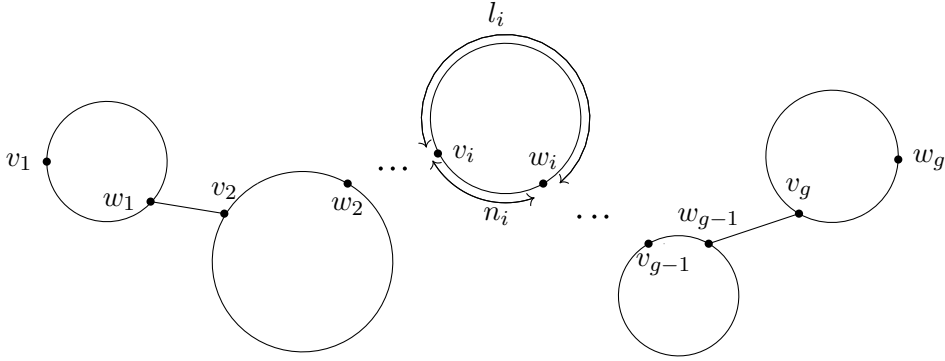
\begin{figure}[h]
\begin{tikzpicture}[scale=0.4]
\draw (-7,0) arc (240:-60:2cm and 2cm);
\draw (-7,0) arc (-120:-60:2cm and 2cm);
\draw (-8,1.73) node[circle, fill=black, scale=0.3, label=left:{$v_1$}]{};
\draw (-4.55,0.4) node[circle, fill=black, scale=0.3, label=left:{$w_1$}]{};
\draw (-4.55,0.4)--(-2.1,0);
\draw (-1,1) arc (120:60:3cm and 3cm);
\draw (-1,1) arc (-240:60:3cm and 3cm);
\draw (-2.1,0) node[circle, fill=black, scale=0.3, label=above:{$v_2$}]{}; 
\draw (2,1) node[circle, fill=black, scale=0.3, label=below:{$w_2$}]{};
\draw (6,1) arc (240:-60:2.5cm and 2.5cm);
\draw (6,1) arc (-120:-60:2.5cm and 2.5cm);
\draw (5,2) node[circle, fill=black, scale=0.3, label=right:{$v_{i}$}]{};
\draw (8.5,1) node[circle, fill=black, scale=0.3, label=above:{$w_{i}$}]{};
\draw (6.95,-0.8) node[ scale=0.3, label=above:{$n_{i}$}]{}; 
\draw (6.95,5.8) node[  scale=0.3, label=above:{$l_{i}$}]{};

\draw[->] [domain=-50:200] plot ({2.8*cos(\x)+7.25}, {2.8*sin(\x)+3.16});
\draw[<-] [domain=-50:200] plot ({2.8*cos(\x)+7.25}, {2.8*sin(\x)+3.16});
\draw[->] [domain=-150:-70] plot ({2.8*cos(\x)+7.25}, {2.8*sin(\x)+3.16});
\draw[<-] [domain=-150:-70] plot ({2.8*cos(\x)+7.25}, {2.8*sin(\x)+3.16});

\draw (12,-1) node[circle, fill=black, scale=0.3]{};
\draw (12,-1) arc (120:60:2cm and 2cm);
\draw (12,-1) arc (-240:60:2cm and 2cm);
\draw (12.5,-1) node[circle, fill=black, scale=0, label=below:{$v_{g-1}$}]{};
\draw (14,-1) node[circle, fill=black, scale=0.3, label=above:{$w_{g-1}$}]{};
\draw (17,0) arc (240:-60:2.2cm and 2.2cm); 
\draw (17,0) arc (-120:-60:2.2cm and 2.2cm);
\draw (17,0) node[circle, fill=black, scale=0.3, label=above:{$v_{g}$}]{};
\draw (20.3,1.8) node[circle, fill=black, scale=0.3, label=right:{$w_g$}]{};
\draw (14,-1)--(17,0); 
\draw (3.1,1.5) node[circle, fill=black, scale=0.15]{};
\draw (3.5,1.5)node[circle, fill=black, scale=0.15]{};
\draw (3.9,1.5)node[circle, fill=black, scale=0.15]{};
\draw (9.7,-0.1) node[circle, fill=black, scale=0.15]{};
\draw (10.2,-0.1)node[circle, fill=black, scale=0.15]{};
\draw (10.6,-0.1)node[circle, fill=black, scale=0.15]{};
\end{tikzpicture}
\caption{A chain of $g$ loops with bridges.} 
\label{fig1}
\end{figure}

\begin{defn}\label{torsion profile}
The \textbf{torsion profile} of $\Gamma$ is a sequence $\underline m=(m_2,...,m_g)$ of $g-1$ integers. If $l_i/n_i$ is a rational number, then $m_i$ is the minimum positive integer such that $m_i\cdot l_i$ is an integer multiple of $l_i+n_i$, otherwise $m_i=0$.
\end{defn}

Note that we omit $m_1$ because it is immaterial to the properties of the divisors of interest. The following notion of a generic chain of loops was introduced in \cite{cools2012tropical} for constructing Brill-Noether general curves:

\begin{defn}\label{generic chain}
We say that $\Gamma$ is \textbf{generic} if  none of the ratios $l_i/n_i$ is equal to the ratio of two positive integers whose sum does not exceed $2g-2$, or equivalently if for each $i$ either $m_i>2g-2$ or $m_i=0$. 
\end{defn}

Let $\lambda$ be a partition, which is a finite, non-increasing sequence of non-negative integers. As in  \cite{pflueger2017special}, we will identify partitions with their Young diagrams in French notation.

\begin{figure}[h]

\begin{tikzpicture}[scale=0.6]

\draw (0,0)--(0,4);
\draw (1,0)--(1,4);
\draw (2,0)--(2,4);
\draw (3,0)--(3,4);
\draw (4,0)--(4,4);
\draw (5,0)--(5,4);
\draw (6,0)--(6,4);
\draw(0,0)--(6,0);
\draw(0,1)--(6,1);
\draw(0,2)--(6,2);
\draw(0,3)--(6,3);
\draw(0,4)--(6,4);
\draw (6.3,0)--(6.3,4);
\draw (7.3,0)--(7.3,3);
\draw (8.3,0)--(8.3,3);
\draw (9.3,0)--(9.3,1);
\draw (6.3,0)--(9.3,0);
\draw (6.3,1)--(9.3,1);
\draw (6.3,2)--(8.3,2);
\draw (6.3,3)--(8.3,3);
\draw (6.3,3.5) node[circle, fill=black, scale=0., label=right:{$\alpha_0$}]{};
\draw (8.3,2.5) node[circle, fill=black, scale=0., label=right:{$\alpha_1$}]{};
\draw (8.5,1.7) node[circle, fill=black, scale=0., label=right:{$\vdots$}]{};
\draw (9.3,0.5) node[circle, fill=black, scale=0., label=right:{$\alpha_r$}]{};
\draw[decorate,decoration={brace,amplitude=3pt},thick] 
    (-0.2,0) node(t_k_unten){} -- 
    (-0.2,4) node(t_k_opt_unten){}; 
\draw (-0.2,2) node[circle, fill=black, scale=0., label=left:{$r+1$}]{};

\draw[decorate,decoration={brace,amplitude=3pt,mirror},thick] 
    (0,-0.2) node(t_k_unten){} -- 
    (6,-0.2) node(t_k_opt_unten){}; 
\draw (3,-0.2) node[circle, fill=black, scale=0., label=below:{$g-d+r$}]{};

\end{tikzpicture}
\caption{The partition as in \cite[Figure 1]{pflueger2017special} associated to $g,d,r,\alpha$ where $r=3$, $d=g-3$ and $\alpha=(0,2,2,3)$.}
\label{}
\end{figure}
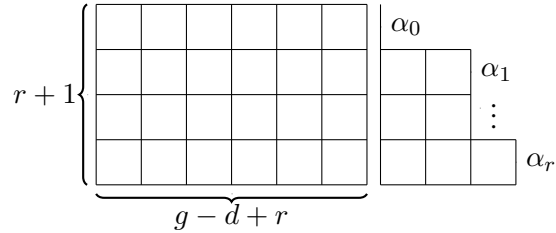

\begin{defn}\label{bn locus partition graph}
Let $P$ be a point on $\G$. The \textbf{Brill-Noether locus} corresponding to a partition $\lambda$ and the marked graph $(\Gamma,P)$ is 
$$W^\lambda(\G,P)=\{D\in\mpic^0(\G):r(D+d'P)\geq r'\mathrm{\ whenever\ } (g-d'+r',r'+1)\in \lambda\}.$$
\end{defn}

Let $\alpha=(\alpha_0,...,\alpha_r)$ be a Schubert index of type $(r,d)$ 
and $\lambda=(\lambda_0,...,\lambda_r)$ be the induced partition where $\lambda_i=(g-d+r)+\alpha_{r-i}$. Then the locus $W^\lambda(\G,P)$ is isomorphic to 
$W^{r,\alpha}_d(\G,\mcP)$
under the Abel-Jacobi map with respect to $dP$. In particular if $\lambda$ is a $(r+1)\times(g-d+r)$ diagram then $W^\lambda(\G,P)$ is isomorphic to $W^r_d(\G)$. 

We next describe the Brill-Noether locus of a partition when $P=w_g$ is the rightmost vertex of $\G$. As in \cite{pflueger2017special} we identify $\lambda=(\lambda_0,...,\lambda_r)$ with the set $$\{(x,y)\in \mathbb Z_{> 0}^2|1\leq x\leq \lambda_{y-1}, 1\leq y\leq r+1\}.$$
\begin{defn}\label{displacement tableaux}
Let $\lambda$ be a partition, and let $\underline m=(m_2,...,m_g)$ be a $(g-1)$-tuple of nonnegative integers. An $\underline m$-\textbf{displacement tableaux} on $\lambda$ is a function $t\colon \lambda\rightarrow \{1,2,...,g\}$ satisfying the following properties:  
 \begin{itemize}
\item[(1)] $t$ is strictly increasing in any given row or column of $\lambda$.
\item[(2)] For any two distinct boxes $(x,y)$ and $(x',y')$ in $\lambda$, if $t(x,y)=t(x',y')$ then $x-y\equiv x'-y'\pmod{m_{t(x,y)}}$.
\end{itemize}
We denote by $t\vdash_ {\underline m} \lambda$ if $t$ is a $\underline m$-displacement tableaux on $\lambda$.
\end{defn}

According to \cite[Theorem 1.3 and Corollary 3.8]{pflueger2017special} if $\G$ is a generic chain of loops and $\underline m$ is its torsion profile, then every $\underline m$-displacement tableaux on $\lambda$ is injective.

\begin{defn}\label{tableux torus}
Let $\underline m$ be the torsion profile of $\G$. Let $t$ be a $\underline m$-displacement tableau on a partition $\lambda$. Denote by $\mT(t)$ the set of divisor classes on $\G$ of the form $$\sum_{i=1}^g \langle \xi_i\rangle_i-gw_g$$ where $\{\xi_j\}_j$ are real numbers such that $\xi_{t(x,y)}\equiv x-y\pmod{m_{t(x,y)}}$ and the symbol $\langle z\rangle_i$ denotes the point on the $i$-th loop that is located $z\cdot l_i$ units clockwise from $w_i$.  
\end{defn}

It follows that $\mT(t)$ is a real torus of dimension $d_t=g-|t(\lambda)|$. Moreover, under the identification $\mpic^0(\G)=\prod_{1\leq i\leq g}\mR/(m_i+l_i)\mZ$ induced by the Abel-Jacobi map \cite[\S 6]{mikhalkin2008tropical}, the torus $\mT(t)$ is the image of a translate of a coordinate $d_t$-plane in $\mR^g$. 
The following is the description of the Brill-Noether locus of $\lambda$ (\cite[Theorem 1.4]{pflueger2017special}).

\begin{prop}\label{dimension}
We have 
$$W^\lambda(\G,w_g)=\bigcup_{t\vdash_{\underline m}\lambda}\mT(t).$$
In particular, if $\G$ is generic then $W^\lambda(\G,w_g)$ is of pure dimension $g-|\lambda|$.
 
\end{prop}

\subsection{Curves with special skeletons and their tropicalizations.}
Let $C$ be a smooth projective curve of genus $g$ over $K$ which has totally split reduction and the skeleton is isometric to $\Gamma$. Let $\tau\colon C^\man\rightarrow \Gamma$ be the retraction map. 
The Jacobian variety of $C$ is totally degenerate in the sense of \cite[\S 6]{gubler2007tropical}.
In other words, $\mathrm{Pic}^0(C)^{\mathrm{an}}$ is isomorphic to $(T_N)^{\man}/L$ where $N$ is a lattice of rank $g$ and $L$ is a discrete subgroup of $T_N(K)$ which maps isomorphically onto a complete lattice of $N_\mR$ under the tropicalization map. Moreover, the induced tropicalization map on $\mpic^0(C)^\man$ is compatible with the retraction to its skeleton, which is canonically identified with $\mpic^0(\Gamma)$ (\cite[\S 6]{baker2014skeleton}): 
$$\begin{tikzcd} 
 C^\man \rar{\alpha_\mcP}\dar{\tau} &\mpic^0(C)^\man\dar{\mtrop}&(T_N)^\man\lar{}\dar{\mtrop}\\ \Gamma\rar{\alpha_P} &\mpic^0(\Gamma)&N_\mR.\lar{}
\end{tikzcd}$$
where $\alpha_\mcP$ and $\alpha_P$ are the Abel-Jacobi maps associated to $\mcP\in C$ and $P\in \Gamma$ with $\tau(\mcP)=P$.

\begin{defn}\label{bn locus partition}
Let $\mcP$ be a point in $C$. As in Definition \ref{bn locus partition graph} the \textbf{Brill-Noether locus} corresponding to a partition $\lambda$ and the marked curve $(C,\mcP)$ is 
$$W^\lambda(C,\mcP)=\{\mcD\in\mpic^0(C):h^0(\mathscr O_C(\mcD+d'\mcP))\geq r'+1\mathrm{\ whenever\ } (g-d'+r',r'+1)\in \lambda\}.$$
\end{defn} 

Let $\alpha=(\alpha_0,...,\alpha_r)$ be a Schubert index of type $(r,d)$ and $\lambda=(\lambda_0,...,\lambda_r)$ be the induced partition as above. Then as in the graph case the locus $W^\lambda(C,\mcP)$ is isomorphic to 
$W^{r,\alpha}_d(C,\mcP)$ 
under the Abel-Jacobi map with respect to $d\mcP$. In particular if $\lambda$ is a $(r+1)\times(g-d+r)$ diagram then $W^\lambda(C,\mcP)$ is isomorphic to $W^r_d(C)$. 

The following theorem is a (partial) summary of \cite[Theorem 1.13 and Theorem 5.1]{pflueger2017special} and \cite[Theorem 1.1]{cartwright2014lifting}.

\begin{thm}\label{summary}
Let $C$ be as above and $\G$ is a generic chain of loops, let $\mcP$ be a point of $C$ that tropicalize to $P=w_g$. Then $\mtrop(W^\lambda(C,\mcP))\subset W^\lambda(\G,P)$ and $W^\lambda(C,\mcP)$ is of pure dimension $g-|\lambda|$. Moreover, if $\lambda$ is an $(r+1)\times(g-d+r)$ diagram then $\mtrop(W^\lambda(C,\mcP))= W^\lambda(\G,P)$. 
\end{thm}
 
As the tropicalization preserves dimension \cite[\S 6]{gubler2007tropical}, in the theorem above, if $\lambda$ is induced by $\alpha$ then the dimension of $W^\lambda(\G,P)$ is equal to the (expected) dimension of $W^\lambda(C,\mcP)$ (or $W^{r,\alpha}_d(C,\mcP)$).

\subsection{Intersection multiplicities in a polytopal domain.} Let $N$ be a lattice of rank $n$. Let $\Delta\subset N_\mR$ be a polytope. We denote $\mcU_\Delta$ the preimage of $\Delta$ in $(T_N)^\man$ under the tropicalization map, then $\mcU_\Delta$ is an affinoid domain in $(T_N)^\man$ by \cite{gubler2007tropical} and called a \textbf{polytopal domain}. Denote $K\langle \mcU_\Delta\rangle$ the corresponding affinoid algebra, whose basic properties can be found in  \cite[Proposition 6.9]{rabinoff2012tropical}.

\begin{defn}\label{polytopal intersection multiplicity}
Let $f_1,...,f_k\in K\langle \mcU_\Delta\rangle$. Let $Y$ be a Zariski-closed analytic subspace of $\mcU_\Delta$ of dimension $n-k$. Denote $Y_i=V(f_i)$ and $Z=Y\cap(\cap_{1\leq i\leq k}Y_i)$. The \textbf{intersection multiplicity} of $Y$ and $Y_i$ at an isolated point $\xi$  of $Z$ is $$i(\xi,Y\cdot Y_1\cdots Y_k;\mcU_\Delta)=\dim_K(\mathcal O_{Z,\xi}).$$ 
If $Z$ is finite we define the \textbf{intersection number} of $Y$ and $Y_1,...,Y_k$ as:
$$i(Y\cdot Y_1\cdots Y_k;\mcU_\Delta)=\sum_{\xi\in Z} \dim_K(\mathcal O_{Z,\xi}).$$ 
\end{defn}

This definition agrees with \cite[Definition 11.4]{rabinoff2012tropical} and is also compatible with the intersection multiplicities of algebraic varieties. We refer to \cite[\S 11]{rabinoff2012tropical} about intersection multiplicities of tropical hypersurfaces in $\Delta$ when $\Delta$ is of maximal dimension, which is compatible with the stable intersection of tropical cycles in $N_\mR$.

\begin{thm}\label{rabinoff intersection}
(\cite[Theorem 11.7]{rabinoff2012tropical}) Let $\Delta$ be a polytope in $N_\mR$ and $f_1,...,f_n\in K\langle \mcU_\Delta\rangle$. Let $Y_i=V(f_i)$ for all $i$ and $w\in \cap_{1\leq i\leq n}\mtrop(Y_i)$ be an isolated point contained in the interior of $\Delta$. Let $Z=\cap_{1\leq i\leq n}Y_i$. Then:
$$\sum_{\xi\in Z,\mtrop(\xi)=w}i(\xi,Y_1\cdots Y_n;\mcU_\Delta)=i(w,\mtrop(Y_1)\cdots \mtrop(Y_n); \Delta)$$
\end{thm}

On the other hand, for two Zariski closed subspace $\mathcal X$ and $\mathcal X'$ of $\mcU_\Delta$ and an isolated point $\xi$ of $ \mathcal X\cap\mathcal X'$, the \textbf{intersection multiplicity} of $\mathcal X$ and $\mathcal X'$ at $\xi$ is defined to be:
$$i(\xi, \mathcal X\cdot\mathcal X';\mcU_\Delta)=\sum_{i\geq 0}(-1)^i\dim_K\mathrm{Tor}_i^{\mathscr O_{\mcU_\Delta,\xi}}(\mathscr O_{\mathcal X,\xi},\mathscr O_{\mathcal X',\xi})$$ 
in \cite[\S 5]{osserman2011lifting}. If $\mathcal X\cap\mathcal X'$ is finite, the \textbf{intersection number} of $\mathcal X$ and $\mathcal X'$ is 
$$i(\mathcal X\cdot\mathcal X';\mcU_\Delta)=\sum_{\xi\in \mathcal X\cap\mathcal X'}i(\xi,\mathcal X\cdot\mathcal X';\mcU_\Delta).$$

\section{Local properties of $W^r_d(C)^{\mathrm{an}}$}  
Let $C$ and $\G$ and $T_N$ be as in \S 2.2 and suppose $\G$ is generic. We prove in this section that $W^r_d(C)$ is locally analytically isomorphic to a polytopal domain at a ``tropically general'' point of $W^r_d(\G)$. 
Before we start, we specify the following notation:

\begin{nota}\label{notation}
Let $P$ be the rightmost vertex of $\G$ and fix $\mcP\in C$ that tropicalize to $P$. Let $e_1,...,e_g\in N$ be the standard basis of $N$ and $e_1',...,e_g'\in M$ the dual basis. For each partition $\lambda$ we write $W^{\lambda}(C)$ (resp. $W^{\lambda} (\G)$) instead of $W^{\lambda}(C,\mcP)$ (resp. $W^{\lambda} (\G,P)$). 
For a polytope $\Delta\subset N_\mR$, if $\Delta$ maps to $\mpic^0(\G)$ isomorphically with image $\overline \Delta$ we denote $\overline \mcU_\Delta$ the preimage of $\overline \Delta$ in $\mpic^0(C)^\man$ under the tropicalization map. 
Let $L_\Delta$ be the subspace of $N_\mR$ that is parallel to $\Delta$ and has the same dimension as $\Delta$. Let $N_\Delta=N\cap L_\Delta$. For a given projection $\pi \colon N\rightarrow N_\Delta$ we denote $\widetilde \mcU_\Delta$ the preimage of $\pi(\Delta)$ in $(T_{N_\Delta})^\man$. For a pure polyhedral complex $\gamma$ in $N_\mR$ denote $\mathrm{relint}(\gamma)$ the union of relative interior of all maximal faces of $\gamma$. 
 \end{nota}

Now let $\lambda$ be the $(r+1)\times(g-d+r)$ diagram. As discussed in Section 2 we have $W^\lambda(C)$ isomorphic to $W^r_d(C)$ and $\mtrop(W^\lambda (C))=W^\lambda(\G)$, and $W^\lambda(\G)$ is a union of translates of the images of the coordinate $\rho$-planes in $N_\mR$, where $\rho=\rho(g,r,d)$.

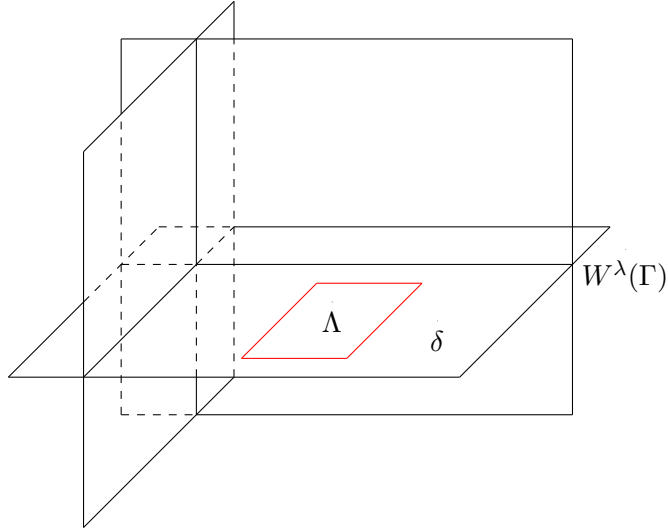
\begin{figure}[h]
\begin{tikzpicture} 


\draw[dashed] (1.5,1.5)--(1.5,5.5);\draw(1.5,5.5)--(1.5,6.5);
\draw[dashed](1.5,1.5)--(2.5,1.5);\draw(2.5,1.5)--(7.5,1.5);
\draw (1.5,6.5)--(7.5,6.5);
\draw (7.5,1.5)--(7.5,6.5);


\draw(0,2)--(6,2);
\draw(0,2)--(1,3);\draw[dashed](1,3)--(2,4);
\draw (6,2)--(8,4);
\draw (3,4)--(8,4);\draw[dashed](2,4)--(3,4);


\draw (1,0)--(1,5);
\draw(1,0)--(3,2);
\draw(1,5)--(3,7);
\draw[dashed](3,2)--(3,6.5);\draw(3,6.5)--(3,7);

\draw(1,2)--(2.5,3.5);\draw[dashed](2.5,3.5)--(3,4);
\draw (2.5,1.5)--(2.5,2);\draw[dashed](2.5,2)--(2.5,3.5);\draw(2.5,3.5)--(2.5,6.5);
\draw (2.5,3.5)--(7.5,3.5);\draw[dashed](1.5,3.5)--(2.5,3.5);
\draw [red](3.1,2.25)--(4.5,2.25);
\draw [red](3.1,2.25)--(4.1,3.25);
\draw [red](4.1,3.25)--(5.5,3.25);
\draw [red](4.5,2.25)--(5.5,3.25);
\draw (4.3,3) node[circle, fill=black, scale=0, label=below:{$\Lambda$}]{};
\draw (5.7,2.75) node[circle, fill=black, scale=0, label=below:{$\delta$}]{};
\draw (8.2,3.7) node[circle, fill=black, scale=0, label=below:{$W^\lambda(\G)$}]{};

\end{tikzpicture}
\end{figure}

Let $\delta$ be a maximal face of $W^\lambda(\G)$. Take a polytope $\Lambda$ such that $\Lambda\subset\mathrm{relint}(\delta)$. Let $\Delta=\Lambda\times I\subset N_\mR$ where $I=[-\e,\e]^{g-\rho}$ such that $\Delta$ maps to  
$\mpic^0(\G)$ isomorphically. 
We then have that $\mcU_\Delta$ is isomorphic to $\overline\mcU_\Delta$. Hence we may consider $W^\lambda (C)^\man$ as a Zariski-closed analytic subspace of the polytopal domain $\mcU_\Delta$. We may assume that $L_\Lambda$ is generated by $e_1,...,e_\rho$. The canonical projection from $N$ to $N_\Lambda$ gives rise to a projection $\pi_\Lambda\colon (T_N)^\man\rightarrow (T_{N_\Lambda})^\man$ which is compatible with the tropicalization map. Denote $W_\Lambda=W^\lambda(C)^\man\cap\mcU_\Lambda$. The argument in \cite[Theorem 4.31]{baker2011nonarchimedean} shows that $\pi_\Lambda\colon W_\Lambda\rightarrow \widetilde\mcU_\Lambda$ is finite, and maps every irreducible component of $W_\Lambda$ surjectively onto $\wmcU_\Lambda$.

\begin{prop}\label{wrd local} 
The map $\pi_\Lambda\colon W_\Lambda\rightarrow \wmcU_\Lambda$ is an isomorphism.
\end{prop}
\begin{proof}
We first show that $\pi_\Lambda$ is of degree one in the sense of \cite[\S 3.27]{baker2011nonarchimedean}. Take $\rho$ general translates of thera divisors $\Theta_\G^1=\mtrop(\Theta_C^1),...,\Theta_\G^\rho=\mtrop(\Theta_C^\rho)$ on $\mpic^0(\G)$ such that $\Lambda\cap(\cap_i\Theta_\G^i)$ is nonempty and consists of finitely many points, where $\Theta_C^i$ are theta divisors on $\mpic^0(C)$. According to \cite[\S 2]{cartwright2014lifting} we may also assume that $\cap_i\Theta_\G^i$ intersects $W^\lambda(\G)$ transversally at $m$ points, where 
$$m=g!\prod_{i=0}^{r}\frac{i!}{(g-d+r+i)!}=i(W^\lambda(C)\cdot \Theta_C^1\cdots \Theta_C^\rho; \mpic^0(C)).
$$  

Since the degree of $\pi_\Lambda$ is preserved under flat base change, we may shrink $\Lambda$ so that $\Lambda\cap(\cap_i\Theta_\G^i)$ consists of exactly one point. Also, take $\e$ small enough such that $\Theta_\G^i\cap\Delta$ is of the form $\widetilde \Theta^i_\G\times I_\e$ where $\widetilde \Theta^i_\G$ is a codimension one polyhedral complex in $\Lambda$. Note that by \cite[\S 6]{wilke2009totally} we know that $K\langle \mcU_\Delta\rangle$ is a UFD, hence we can take $f_i\in K\langle\mcU_\Delta\rangle$ to be the function that defines $(\Theta_C^i)^\man$ in $\mcU_\Delta$. 

According to \cite[\S 8]{rabinoff2012tropical} there is a Laurent polynomial $f_i'$ which is a sum of monomials in $f_i$ such that $\mtrop(V(f_i'))\cap\Delta=\mtrop(V(f_i))=\widetilde \Theta^i_\G\times I_\e$. Moreover for all $w\in \Delta$ the monomials in $f_i$ which obtains minimal $w$-weight is the same as those in $f_i'$. Let $A=\{{u_1},..., {u_k}\}\subset M$ be the set of vertices of the Newton complex of $f_i'$ corresponding to the maximal faces of $\Delta$, whose polyhedral complex structure is induced by $\widetilde \Theta^i_\G\times I_\e$. We must have that $A$ is contained in a $\rho$ dimensional plane in $M_\mR$ that is parallel to the one generated by $e_1',...,e_\rho'$. We may assume that $A$ is contained in the sublattice generated by $ e'_1,...,e'_\rho$. Consequently, if $g_i=\sum x^{u_i}$ and $h_i=f_i-g_i$, then for every $a\in K$ with $\mathrm{val}(a)\geq 0$ we have $\mtrop(V(g_i+ah_i))=\mtrop(V(f_i))$ (with the same multiplicities, which are all ones by \cite[Theorem 3.1]{cartwright2014lifting}). Moreover, $g_i$ is contained in $K\langle\wmcU_\Lambda\rangle$. 

We next denote by $\mB_K^1$ the unit ball in $(\mG_m)_K^\man$ with coordinate ring $K\langle t\rangle$, and prove the following lemma:

\begin{lem}\label{intersection family}
Let $W$ be a Cohen-Macaulay Zariski-closed analytic subspace of $\mcU_\Delta$ of pure codimension $k$. Let $l_1,...,l_k\in K\langle\mcU_\Delta\rangle\times K\langle t\rangle$ be global sections on $\mcU_\Delta\times \mB_K^1$. Let $Y_i$ be the subspace of $\mcU_\Delta \times \mB_K^1$ defined by $l_i$. For $t\in \mB_K^1$ let $Y_i(t)=\pi^{-1}(t)\cap Y_i$ where $\pi$ is the projection from $\mcU_\Delta \times \mB_K^1$ to $\mB_K^1$. Suppose $\mtrop(W)\cap(\cap_i\mtrop(Y_i(t)))$ is finite and contained in $\Delta^\circ$ for all $i$ and $t\in |\mB_K^1|$. Then the intersection number $i(W\cdot Y_1(t)\cdots Y_k(t);\mcU_\Delta)$ is constant on $|\mB_K^1|$.
\end{lem} 
\begin{proof}
We proceed by showing that the analytic space $\widetilde W=(W\times \mB_K^1)\cap(\cap_i Y_i)$ is finite and flat over $\mB_K^1$, hence every fiber has the same length. 

It is obvious that $\widetilde W$ has finite fiber. To show it is proper over $\mB_K^1$ we use the ideas in \cite[\S 4.9]{osserman2011lifting}. Since all analytic space appeared are affinoid, hence compact Hausdorff, we have that $\pi$ is compact on $\widetilde W$ and separated. On the other hand, let $\pi'\colon \mcU_\Delta\times \mB_K^1\rightarrow \mcU_\Delta$ be the other projection. By \cite[Lemma 4.14]{osserman2011lifting} we have $$\widetilde W\subset (\mtrop\circ\pi')^{-1}(\Delta^\circ)\subset \mathrm{Int}(\mcU_\Delta\times \mB_K^1/\mB_K^1).$$ According to the sequence of morphisms $\widetilde W\rightarrow \mcU_\Delta\times \mB_K^1\rightarrow \mB_K^1$ we have 
$$\mathrm{Int}(\widetilde W/\mB_K^1)=\mathrm{Int}(\widetilde W/\mcU_\Delta\times \mB_K^1)\cap \mathrm{Int}(\mcU_\Delta\times \mB_K^1/\mB_K^1)=\mathrm{Int}(\widetilde W/\mcU_\Delta\times \mB_K^1)=\widetilde W.$$
Hence $\pi$ is boundaryless on $\widetilde W$. Consequently $\pi$ is proper, and hence finite on $\widetilde W$.

Now the flatness of $\pi$ follows from \cite[Exercise 1.2.12]{liu2002algebraic} and induction. Note that the finiteness of fibers of $\pi$ and the Cohen-Macaulay-ness of $W$ ensures that each $l_i$ is not a zero divisor on $W\cap Y_1(t)\cap\cdots\cap Y_{i-1}(t)$ for all $t$.
\end{proof}
We return to the proof of Proposition \ref{wrd local}. In Lemma \ref{intersection family} let $W=W_\Lambda$ and $l_i=g_i+th_i$. It follows that (set $t=0$)
$$i(W_\Lambda\cdot \prod_{i=1}^{\rho} (\Theta_C^i)^\man;\mcU_\Delta)=i(W_\Lambda\cdot \prod_{i=1}^{\rho} V(f_i);\mcU_\Delta)= i(W_\Lambda\cdot \prod_{i=1}^{\rho} V(g_i);\mcU_\Delta)=m_\Lambda\cdot i( \prod_{i=1}^{\rho} V(g_i);\wmcU_\Lambda)$$
where the last equation is the projection formula in \cite[Proposition 2.10]{gubler1998local} and $m_\Lambda$ is the degree of $\pi_\Lambda$. By Theorem \ref{rabinoff intersection} we have $$i( \prod_{i=1}^{\rho} V(g_i);\wmcU_\Lambda)=i( \prod_{i=1}^{\rho}\mtrop( V(g_i));\Lambda)=1,$$ therefore $i(W_\Lambda\cdot \prod_{i=1}^{\rho} (\Theta_C^i)^\man;\mcU_\Delta)=m_\Lambda$.

Now for all $w_j\in W^\lambda(\G)\cap(\cap_i \Theta_\G^i)$ where $1\leq j\leq m$ we pick a polytope $\Lambda_j$ as above and get a degree $m_{\Lambda_i}$ of the corresponding projection map, which yields 
$$\sum_{j=1}^{m}m_{\Lambda_j}= i(W^\lambda(C)\cdot \Theta_C^1\cdots \Theta_C^\rho; \mpic^0(C))=m.$$
Note that the first equality follows from the fact that the $K$-dimension of the local ring of $W^\lambda (C)\cap(\cap_i \Theta^i_C)$ at a point is equal to that of $W^\lambda (C)^\man\cap(\cap_i (\Theta^i_C)^\man))$. Hence we must have $m_{\Lambda_j}=1$ for all $j$. Therefore $\pi_\Lambda$ is of degree one.

It follows that $W_\Lambda$ is irreducible and generically reduced. However, $W_\Lambda$ is Cohen-Macaulay since $W^\lambda(C)$ is, so it is everywhere reduced, hence integral. Now $\pi_\Lambda$ induces a finite morphism of degree one between integral domains whose source $K\langle \wmcU_\Lambda\rangle$ is normal, it must be an isomorphism.
\end{proof}
\begin{rem}\label{local property remark} 
An algebraic analogue of Proposition \ref{wrd local} is that a reduced (or Cohen-macaulay) closed subscheme $Z$ of $T_N$ is local analytically isomorphic to a torus at a point that tropicalize to the relative interior of a maximal face of $\mtrop(Z)$ of multiplicity one, see for example \cite[Lemma 6.2]{he2016generalization}.
\end{rem} 


\section{Lifting tropical intersections in a polyhedral domain.}
Let $N$ be an arbitrary lattice of rank $n$ as in Theorem \ref{lifting polyhedral domain} and $T_N$ the induced torus. Let $\Delta$ be a polytope of maximal dimension in $N_\mR$. In this section we use Osserman and Rabinoff's continuity theorem \cite[\S 5]{osserman2011lifting} of analytic intersection numbers to prove Theorem \ref{lifting polyhedral domain}. Let $\mcU_0\subset T_N^\man$ be the preimage of the origin in $N_\mR$. Then $\mcU_0$ acts on $\mcU_\Delta$. Denote the action by $\mu\colon \mcU_0\times \mcU_\Delta\rightarrow \mcU_\Delta$ and let $\pi\colon \mcU_0\times\mcU_\Delta\rightarrow\mcU_0$ be the projection. This gives an isomorphism:
$$(\pi,\mu)\colon \mcU_0\times \mcU_\Delta\rightarrow  \mcU_0\times \mcU_\Delta.$$

The following lemma is a consequence of \cite[Proposition 5.8]{osserman2011lifting}:

\begin{lem}\label{analytic continuity} 
Let $\pi$ be as above. Let $\mathcal Y,\mathcal Y'\subset \mcU_0\times\mcU_\Delta$ be Zariski-closed subspaces, flat over $\mcU_0$, such that $\mathcal Y\cap\mathcal Y'$ is finite over $\mcU_0$. Then the map 
$$s\mapsto i(\mathcal Y_s\cdot\mathcal Y_s';\mcU_\Delta)\quad\colon\quad |\mcU_0|\rightarrow \mZ$$ is constant on $\mcU_0$, where $\mathcal Y_s$ and $\mathcal Y_s'$ are fibers of $\pi$.
\end{lem}
 
We refer to \cite{ducros2011flatness} or \cite[\S 5]{osserman2011lifting} about the notion of flatness for analytic spaces, which is preserved under composition and change of base. Any analytic space is flat over $K$. We now prove Theorem 1.2:

\begin{proof}[Proof of Theorem \ref{lifting polyhedral domain}] Let $\mathcal X,\mathcal X'\subset \mcU_\Delta$ be Zariski closed analytic subspaces of $\mcU_\Delta$ of pure dimension. We first assume that $\dim (\mathcal X)+\dim(\mathcal X')=n$, hence $\mathcal X\cap\mathcal X'$ is finite. As the statement is local, we may also assume that $\mtrop(\mathcal X)\cap\mtrop(\mathcal X')$ contains only one point $w$ that lies in $\Delta^\circ$. It suffices to show that $\mathcal X\cap\mathcal X'$ is nonempty.

In Lemma \ref{analytic continuity}, let $\mathcal Y=(\pi,\mu)( \mcU_0\times \mathcal X)$ and $\mathcal Y'= \mcU_0\times \mathcal X'$ where $\pi$ and $\mu$ are as above. Then both $\mathcal Y$ and $\mathcal Y'$ are flat over $\mcU_0$. On the other hand, the argument in Lemma \ref{intersection family} shows that $\mathcal Y\cap\mathcal Y'$ is finite over $\mcU_0$, as $\mtrop(\mathcal Y_s)\cap\mtrop(\mathcal Y_s')=\mtrop(\mathcal X)\cap\mtrop(\mathcal X')$ is finite for every $s\in |\mcU_0|$. Therefore $i(\mathcal Y_s\cdot\mathcal Y_s';\mcU_\Delta)$ is constant on $|\mcU_0|$ by Lemma \ref{analytic continuity}. Take $\xi\in|\mathcal X|$ and $\xi'\in|\mathcal X'|$ such that $\mtrop(\xi)=\mtrop(\xi')=w$. Take also  
 $t\in |\mcU_0|$ such that $t(\xi)=\xi'$. Then $\mathcal Y_t\cap\mathcal Y'_t=t(\mathcal X)\cap \mathcal X'$ contains $\xi'$, hence $i(\mathcal Y_t\cdot\mathcal Y_t';\mcU_\Delta)>0$. Thus $i(\mathcal Y_s\cdot\mathcal Y_s';\mcU_\Delta)>0$ for all $s\in |\mcU_0|$. Taking $s$ to be the identity in $\mcU_0$ implies that $\mathcal X\cap\mathcal X'=\mathcal Y_s\cap\mathcal Y_s'$ is nonempty. Thus $w\in \mtrop(\mathcal X)\cap\mtrop(\mathcal X')$

This situation easily generalizes to intersections of three or more subschemes as in \cite[\S 5.2]{osserman2013lifting}. Namely, we have the following lemma:

\begin{lem}\label{multiple finte intersection}
Let $\mathcal X_1,...,\mathcal X_m$ be Zariski-closed subspaces of $\mcU_\Delta$ of (pure) codimension $d_1,...,d_m$ respectively, where $d_1+\cdots +d_m=n$. Suppose $\mtrop(\mathcal X_1)\cap\cdots\cap\mtrop(\mathcal X_m)$ is finite and contained in $\Delta^\circ$. Then 
$$\mtrop(\mathcal X_1)\cap\cdots\cap\mtrop(\mathcal X_m)=\mtrop(\mathcal X_1\cap\cdots\cap\mathcal X_m).$$

\end{lem}

Now suppose $\mtrop(\mathcal X)\cap\mtrop(\mathcal X')$ has dimension $l>0$. For any $G$-rational point $v\in \mtrop(\mathcal X)\cap \mtrop(\mathcal X')\cap \Delta^\circ$ we can find a Zariski-closed subspace $\mathcal Z$ of $\mcU_\Delta$ of codimension $l$ such that $\mtrop(\mathcal Z)$ contains $v$ and intersect properly with $\mtrop(\mathcal X)\cap\mtrop(\mathcal X')$ near $v$. Hence Lemma \ref{multiple finte intersection} implies that $$v\in\mtrop(\mathcal Z\cap \mathcal X\cap\mathcal X')\subset \mtrop(\mathcal X\cap\mathcal X').$$
As $G$-rational points are dense in $\mtrop(\mathcal X)\cap\mtrop(\mathcal X')$ this implies that $\mtrop(\mathcal X)\cap\mtrop(\mathcal X')\cap\Delta^\circ=\mtrop(\mathcal X\cap\mathcal X')\cap\Delta^\circ.$
\end{proof}

As mentioned in the proof above, we can also state Theorem \ref{lifting polyhedral domain} for the intersection of more than two analytic subspaces of $\mcU_\Delta$:

\begin{cor}\label{multiple intersection}
Let $\mathcal X_1,...,\mathcal X_m$ be Zariski-closed subspaces of $\mcU_\Delta$ of pure dimensions whose tropicalizations intersect properly. Then 
$$\mtrop(\mathcal X_1)\cap\cdots\cap\mtrop(\mathcal X_m)\cap\Delta^\circ=\mtrop(\mathcal X_1\cap\cdots\cap\mathcal X_m)\cap\Delta^\circ.$$
\end{cor}

\begin{rem}\label{polyhedral complex}
As the statement in Corollary \ref{multiple intersection} is local, it is still true if we replace $\Delta$ by the support of a polyhedral complex with integral $G$-affine faces. In particular, we have 
$$\mtrop(\mathcal X_1)\cap\cdots\cap\mtrop(\mathcal X_m)=\mtrop(\mathcal X_1\cap\cdots\cap\mathcal X_m)$$
if $\mathcal X_1,...,\mathcal X_m$ are Zariski-closed analytic subspaces of $T_N^\man$ with proper tropical intersections. 
\end{rem}

It is necessary to only consider the interior of $\Delta$ in Theorem \ref{lifting polyhedral domain} or Corollary \ref{multiple intersection}. See the example below.

\begin{ex}\label{interior points}

$$\begin{tikzpicture}
\fill[fill=lightgray] (0,0)--(-0,1)--(1,1)--(1,0);
\draw (0,0)--(-0,1)--(1,1)--(1,0);

\draw[red] (0,0)--(0,2);
\draw[red] (0,0)--(2,0);
\draw[red] (0,0.05)--(-1.5,-1.45);

\draw[blue] (-1.5,-1.5)--(2,2);

\draw (0,0) node[circle, fill=black, scale=0.3, label=left:{$O$}]{};
\draw (0.5,-0.05) node[circle, fill=black, scale=0, label=above:{$\Delta$}]{};
\draw (0,2) node[circle, fill=black, scale=0, label=above:{$\mtrop(\mathcal X)$}]{};
\draw (2.5,2) node[circle, fill=black, scale=0, label=above:{$\mtrop(\mathcal X')$}]{};

\end{tikzpicture}
$$
Let $\mathcal X$ and $\mathcal X'$ be two curves in $(K^*)^2$ defined by $x+y+1=0$ and $x+y=0$ respectively. Let $\Delta=[0,1]^2$. Then $\mtrop(\mathcal X)\cap\mtrop(\mathcal X')\cap\Delta$ is the origin, hence $\mtrop(\mathcal X)$ and $\mtrop(\mathcal X')$ intersect properly in $\Delta$. But $\mathcal X\cap\mathcal X'\cap\mcU_\Delta$ is empty.

\end{ex}

\section{Lifting divisors with imposed ramification}
In this section we prove Theorem \ref{introduction}. We will use the notations in Notation \ref{notation}. Let $\alpha=(\alpha_0,...,\alpha_r)$ be a Schubert index of type $(d,r)$. Let $\lambda=(\lambda_0,...,\lambda_r)$ be the induced partition. Hence $\lambda_i=g-d+r+\alpha_{r-i}$.

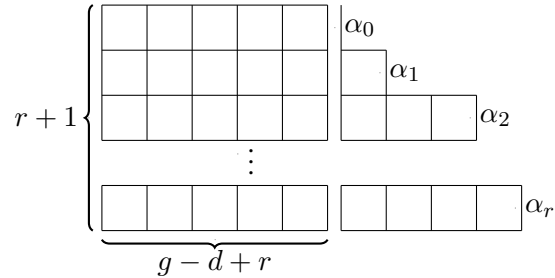
$$
\begin{tikzpicture}[scale=0.6]

\draw (0,2)--(0,5);
\draw (1,2)--(1,5);
\draw (2,2)--(2,5);
\draw (3,2)--(3,5);
\draw (4,2)--(4,5);
\draw (5,2)--(5,5);

\draw (5.3,0)--(5.3,1);
\draw (5.3,2)--(5.3,5);

\draw (6.3,2)--(6.3,4);
\draw(0,0)--(5,0);
\draw(0,1)--(5,1);
\draw(0,2)--(5,2);
\draw(0,3)--(5,3);
\draw(0,4)--(5,4);\draw (5.3,4)--(6.3,4);
\draw(0,5)--(5,5);

\draw (7.3,2)--(7.3,3);
\draw (8.3,2)--(8.3,3);
\draw (9.3,0)--(9.3,1);
\draw (5.3,0)--(9.3,0);
\draw (5.3,1)--(9.3,1);
\draw (5.3,2)--(8.3,2);
\draw (5.3,3)--(8.3,3);

\draw (0,0)--(0,1);
\draw (1,0)--(1,1);
\draw (2,0)--(2,1);
\draw (3,0)--(3,1);
\draw (4,0)--(4,1);
\draw (5,0)--(5,1);
\draw (6.3,0)--(6.3,1);
\draw (7.3,0)--(7.3,1);
\draw (8.3,0)--(8.3,1);
\draw (3,1.7) node[circle, fill=black, scale=0., label=right:{$\vdots$}]{};

\draw[decorate,decoration={brace,amplitude=3pt},thick] 
    (-0.2,0) node(t_k_unten){} -- 
    (-0.2,5) node(t_k_opt_unten){}; 
\draw (-0.3,2.5) node[circle, fill=black, scale=0., label=left:{$r+1$}]{};

\draw[decorate,decoration={brace,amplitude=3pt,mirror},thick] 
    (0,-0.2) node(t_k_unten){} -- 
    (5,-0.2) node(t_k_opt_unten){}; 
\draw (2.5,-0.2) node[circle, fill=black, scale=0., label=below:{$g-d+r$}]{};

\draw (5.15,4.5) node[circle, fill=black, scale=0., label=right:{$\alpha_0$}]{};
\draw (6.15,3.5) node[circle, fill=black, scale=0., label=right:{$\alpha_1$}]{};
\draw (8.15,2.5) node[circle, fill=black, scale=0., label=right:{$\alpha_2$}]{};
\draw (9.15,0.5) node[circle, fill=black, scale=0., label=right:{$\alpha_r$}]{};

\end{tikzpicture}
$$
After translating every divisor class on $C$ (resp. $\G$) of degree $d$ to its image in $\mpic^0(C)$ (resp. $\mpic^0(\G)$) under the Abel-Jacobi map induced by $d\mcP$ (resp $dP$) we may assume that the ramification is imposed by $\lambda$ (instead of $\alpha$). Hence it remains to prove the following:

\begin{thm}\label{lifting ramification}
We have $\mtrop(W^\lambda(C))=W^\lambda(\G)$.
\end{thm}

Let $\lambda_j$ be the partition corresponding to the $(r+1-j)\times(g-d+r+\alpha_j)$ diagram. Then $W^{\lambda_j}(C)$ is isomorphic to $W^{r-j}_{d-\alpha_j-j}(C)$, and $W^\lambda(C)=\bigcap_{0\leq i\leq r}W^{\lambda_i}(C)$, while $W^{\lambda_j}(\G)$ is isomorphic to $W^{r-j}_{d-\alpha_j-j}(\G)$, and $W^\lambda(\G)=\bigcap_{0\leq i\leq r}W^{\lambda_i}(\G)$. Let  $\lambda^j$ be the union of $\lambda_1,...,\lambda_j$ and $W_j(C)=\bigcap_{0\leq i\leq j}W^{\lambda_i}(C)=W^{\lambda^j}(C)$ and $W_j(\G)=\bigcap_{0\leq i\leq j}W^{\lambda_i}(\G)=W^{\lambda^j}(\G)$  for $0\leq j\leq r$. 

Let also $\mu_{j}$ be the partition corresponding to the $(r-j)\times(g-d+r+\alpha_j)$ diagram (this is the intersection of $\lambda^j$ and $\lambda_{j+1}$). As above we have $W^{\mu_{j}}(C)$ isomorphic to $W^{r-j-1}_{d-\alpha_j-j-1}(C)$, and $W^{\mu_{j}}(\G)$ isomorphic to $W^{r-j-1}_{d-\alpha_j-j-1}(G)$. Moreover, we have $W_j(C)\subset W^{\mu_j}(C)$ and $W^{\lambda_{j+1}}(C)\subset W^{\mu_j}(C)$ and $W_j(\G)\subset W^{\mu_j}(\G)$ and $W^{\lambda_{j+1}}(\G)\subset W^{\mu_j}(\G)$.

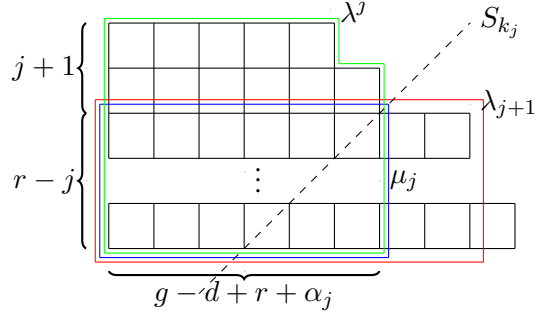
$$
\begin{tikzpicture}[scale=0.6]

\draw (0,2)--(0,5);
\draw (1,2)--(1,5);
\draw (2,2)--(2,5);
\draw (3,2)--(3,5);
\draw (4,2)--(4,5);
\draw (5,2)--(5,5);

\draw (6,2)--(6,4);
\draw(0,0)--(6,0);
\draw(0,1)--(6,1);
\draw(0,2)--(6,2);
\draw(0,3)--(6,3);
\draw(0,4)--(6,4);
\draw(0,5)--(5,5);

\draw (7,2)--(7,3);
\draw (8,2)--(8,3);
\draw (9,0)--(9,1);
\draw (6,0)--(9,0);
\draw (6,1)--(9,1);
\draw (6,2)--(8,2);
\draw (6,3)--(8,3);

\draw (0,0)--(0,1);
\draw (1,0)--(1,1);
\draw (2,0)--(2,1);
\draw (3,0)--(3,1);
\draw (4,0)--(4,1);
\draw (5,0)--(5,1);
\draw (6,0)--(6,1);
\draw (7,0)--(7,1);
\draw (8,0)--(8,1);
\draw (3,1.7) node[circle, fill=black, scale=0., label=right:{$\vdots$}]{};

\draw [red] (-0.3,-0.3)--(8.3,-0.3);
\draw [red] (-0.3,-0.3)--(-0.3,3.3);
\draw [red] (-0.3,3.3)--(8.3,3.3);
\draw [red] (8.3,-0.3)--(8.3,3.3);
\draw (8,3.2) node[circle, fill=black, scale=0., label=right:{$\lambda_{j+1}$}]{};

\draw [blue] (-0.2,-0.2)--(6.2,-0.2);
\draw [blue] (-0.2,-0.2)--(-0.2,3.2);
\draw [blue] (-0.2,3.2)--(6.2,3.2);
\draw [blue] (6.2,-0.2)--(6.2,3.2);
\draw (6,1.5) node[circle, fill=black, scale=0., label=right:{$\mu_{j}$}]{};

\draw [green] (-0.1,-0.1)--(6.1,-0.1);
\draw [green] (-0.1,-0.1)--(-0.1,5.1);
\draw [green] (-0.1,5.1)--(5.1,5.1);
\draw [green] (6.1,4.1)--(6.1,-0.1);
\draw [green] (6.1,4.1)--(5.1,4.1);
\draw [green] (5.1,5.1)--(5.1,4.1);
\draw (4.9,5.3) node[circle, fill=black, scale=0., label=right:{$\lambda^{j}$}]{};

\draw[decorate,decoration={brace,amplitude=3pt},thick] 
    (-0.5,3) node(t_k_unten){} -- 
    (-0.5,5) node(t_k_opt_unten){}; 
\draw (-0.5,4) node[circle, fill=black, scale=0., label=left:{$j+1$}]{};

\draw[decorate,decoration={brace,amplitude=3pt},thick] 
    (-0.5,0) node(t_k_unten){} -- 
    (-0.5,3) node(t_k_opt_unten){}; 
\draw (-0.5,1.5) node[circle, fill=black, scale=0., label=left:{$r-j$}]{};

\draw [dashed] (2,-1)--(8,5);
\draw (8,5) node[circle, fill=black, scale=0., label=right:{$S_{k_{j}}$}]{};

\draw[decorate,decoration={brace,amplitude=3pt,mirror},thick] 
    (0,-0.5) node(t_k_unten){} -- 
    (6,-0.5) node(t_k_opt_unten){}; 
\draw (3,-0.5) node[circle, fill=black, scale=0., label=below:{$g-d+r+\alpha_j$}]{};

\end{tikzpicture}
$$

In order to prove the Theorem above, we first show the following lemma:

\begin{lem}\label{intersect interior}
$W_j(\G)$ and $W^{\lambda_{j+1}}(\G)$ intersect properly in $W^{\mu_j}(\G)$, and there is an open dense subset $U_j$ of $W_j(\G)\cap W^{\lambda_{j+1}}(\G)=W_{j+1}(\Gamma)$ which is contained in $\mathrm{relint}(W^{\mu_j}(\G))$.
\end{lem}
\begin{proof} The properness follows directly from dimension counting (Proposition \ref{dimension}), as $\lambda^{j+1}$ is the union of $\lambda^j$ and $\lambda_{j+1}$ while $\mu_j$ is the intersection of $\lambda^j$ and $\lambda_{j+1}$.

For the second conclusion it suffices to show that every real torus in $W_{j+1}(\G)$ is contained in exactly one torus in $W^{\mu_j}(\G)$. Take two tori $\mT(t)\subset W_{j+1}(\Gamma)$ and $\mT(t')\subset W^{\mu_j}(\G)$, where $t$ and $t'$ are $\underline m$-displacement tableux on $\lambda^{j+1}$ and $\mu_j$ respectively, such that $\mT(t)\subset \mT(t')$. We claim that $t'=t|_{\mu_j}$. 

It is easy to see that $t'(\mu_j)\subset t(\lambda^{j+1})$. On the other hand, let $S_k=\{(x,y)|x-y=k\}$ for all $k\in \mZ$. If $t(x,y)=t'(x',y')$, then $x-y\equiv x'-y'\pmod {m_{t(x,y)}}$, hence $x-y=x'-y'$ by the generality of $\G$. It follows that $t(\mu_j\cap S_k)\subset t'(\lambda^{j+1}\cap S_k)$ for all $k$. In particular, let $k_j=g-d+\alpha_j+j$, we have $$t(\mu_j\cap S_{k_j})= t'(\lambda^{j+1}\cap S_{k_j})$$ since $\mu_j\cap S_{k_j}= \lambda^{j+1}\cap S_{k_j}$. Therefore $t|_{\mu_j\cap S_{k_j}}= t'|_{\mu_{j}\cap S_{k_j}}$ as both $t$ and $t'$ are strictly increasing along rows and columns. 

It follows that $t'(r-j+k_j,r-j+1)\not\in t(\mu_j\cap S_{k_j-1})$, since this number is bigger than all numbers in $t'(\lambda^{j+1}\cap S_{k_j})=t(\mu_j\cap S_{k_j})$, thus greater that numbers in $ t(\mu_j\cap S_{k_j-1})$. It then follows that $t(\mu_j\cap S_{k_j-1})= t'(\mu_j\cap S_{k_j-1})$, therefore $t|_{\mu_j\cap S_{k_j-1}}= t'|_{\mu_{j}\cap S_{k_j-1}}$. 
Now one can check by induction that $t|_{\mu_j\cap S_{k}}= t'|_{\mu_{j}\cap S_{k}}$ for all $k\leq k_j$. 
Same argument shows that 
$t|_{\mu_j\cap S_{k}}= t'|_{\mu_{j}\cap S_{k}}$ 
for all $k\geq k_j$.
\end{proof}

\begin{proof}[Proof of Theorem \ref{lifting ramification}]
We prove by induction that $\mtrop(W_k(C))=W_k(\G)$ for all $0\leq k\leq r$. The $k=0$ case is in Theorem \ref{summary}. Now assume $\mtrop(W_j(C))=W_j(\G)$, we need to show that 
\begin{equation}\mtrop(W_j(C)\cap W^{\lambda_{j+1}}(C))=\mtrop(W_{j+1}(C))=W_{j+1}(\G).
\end{equation}
 
Let $U_j$ be as in Lemma \ref{intersect interior} and fix $w\in U_j$. As we only care about the local geometry near $w$, we may assume all Brill-Noether loci corresponding to $(C,\mcP)$ (resp. $(\G,P)$) are contained in a polytopal domain (resp. polytope) in $T_N$ (resp. $N_\mR$). We may also assume that $w$ is the origin. Take a polytope $\Lambda \subset \mathrm{relint}(W^{\mu_j}(\G))$ such that $w\in \mathrm{relint}(\Lambda)$. According to Proposition \ref{wrd local} we have the following commutative diagram:
$$\begin{tikzcd} 
 W^{\mu_j}(C)^\man\cap \mcU_\Lambda \rar{\mtrop}\dar{ \pi_\Lambda} &\Lambda\dar{\pi} \\ \wmcU_\Lambda\rar{\mtrop} &\pi(\Lambda)
\end{tikzcd}$$
where both vertical arrows are isomorphisms induced by the natural projection from $N$ to $N_\Lambda$ as in loc.cit..
 
Denote $$W_\Lambda^{\lambda_{j+1}}=W^{\lambda_{j+1}}(C)^\man\cap\mcU_\Lambda\mathrm{\ and \ }W_{\Lambda, j}=W_j(C)^\man\cap \mcU_\Lambda.$$According to Lemma \ref{intersect interior} $\mtrop( \pi_\Lambda(W_\Lambda^{\lambda_{j+1}}))$ and $\mtrop(\pi_\Lambda(W_{\Lambda,j}))$ intersect properly in $\pi(\Lambda)$, which is a polytope of maximal dimensional in $(N_\Lambda)_\mR$ that contains $\pi(w)$ as an interior point. Hence Theorem \ref{lifting polyhedral domain} implies that $\pi(w)\in\mtrop(\pi_\Lambda(W_\Lambda^{\lambda_{j+1}})\cap \pi_\Lambda(W_{\Lambda,j}))$, and that $$w\in\mtrop(W_\Lambda^{\lambda_{j+1}}\cap W_{\Lambda,j})\subset \mtrop(W_j(C)\cap W^{\lambda_{j+1}}(C)).$$

As $U_j$ is dense in $W_{j+1}(\G)$ and $U_j\subset\mtrop(W_j(C)\cap W^{\lambda_{j+1}}(C))$, we have $W_{j+1}(\G)\subset\mtrop(W_j(C)\cap W^{\lambda_{j+1}}(C))$. This proves (2), as the other direction of containment is trivial.
\end{proof}

\bibliographystyle{amsalpha}
\bibliography{1}
\end{document}